\documentclass[reqno]{amsart}

\usepackage{amscd,url,fancyvrb}

\newcommand{\NN}{\mathbf{N}} 
\newcommand{\ZZ}{\mathbf{Z}} 
\newcommand{\QQ}{\mathbf{Q}} 
\newcommand{\RR}{\mathbf{R}} 
\newcommand{\FF}{\mathbf{F}} 
\newcommand{\PP}{\mathbf{P}} 

\providecommand{\card}[1]{\lvert#1\rvert} 
\providecommand{\HdR}{H_{\text{dR}}}   	      	
\providecommand{\Hrig}{H_{\text{rig}}} 	      	
\providecommand{\BigOh}{O}          		
\providecommand{\SoftOh}{\tilde{O}} 		

\DeclareMathOperator{\Frob}{F}          
\DeclareMathOperator{\Gal}{Gal} 	
\DeclareMathOperator{\ord}{ord} 	
\DeclareMathOperator{\fCoKer}{coker}    
\DeclareMathOperator{\Spec}{Spec} 	

\newtheorem{thm}{Theorem}[section]

\newtheorem{prop}[thm]{Proposition}

\newtheorem{defn}[thm]{Definition}

\newtheorem{rem}[thm]{Remark}

\newtheorem{assump}{Assumption}

\title{Counting points on curves using a map to $\mathbf{P}^1$.}
\author{Jan Tuitman}
\address{KU Leuven,
         Departement Wiskunde,
         Celestijnenlaan 200B,
         3001 Leuven,
         Belgium}
\email{jan.tuitman@wis.kuleuven.be}

\begin{document}

\begin{abstract}
We introduce a new algorithm to compute the zeta function of a curve over a finite field. 
This method extends Kedlaya's algorithm to a very general class of curves using a map
to the projective line. We develop all the necessary bounds, analyse the complexity of 
the algorithm and provide some examples computed with our implementation.
\end{abstract}

\maketitle

\section{Introduction}

Let $\FF_q$ denote the finite field of characteristic $p$ and cardinality $q=p^n$. Moreover, let $\QQ_p$ denote
the field of $p$-adic numbers and $\QQ_q$ its unique unramified extension of degree $n$. As usual, let 
$\sigma \in \Gal(\QQ_q/\QQ_p)$ denote the unique element that lifts the $p$-th power Frobenius map
on $\FF_q$. Finally, let $\ZZ_q$ denote the ring of integers of $\QQ_q$, so that $\ZZ_q/p\ZZ_q \cong \FF_q$. 
Suppose that $X$ is a smooth proper algebraic curve of genus $g$ over $\FF_q$. Recall that the zeta function 
of $X$ is defined as
\[
Z(X,T)=\exp\left(\sum_{i=1}^{\infty} \card{X(\FF_{q^i})} \frac{T^i}{i} \right).
\]
It follows from the Weil conjectures that $Z(X,T)$ is of the form
\[
\frac{\chi(T)}{(1-T)(1-qT)},
\]
with $\chi(T) \in \ZZ[T]$ a polynomial of degree $2g$, the inverse roots of which have complex absolute value $q^{\frac{1}{2}}$ 
and are permuted by the map $t \rightarrow q/t$. Moreover, by the Lefschetz formula for rigid cohomology, we have
that
\[
\chi(T)=\det\left(1-T \Frob_p^n|\Hrig^1(X)\right),
\]
where $\Frob_p$ denotes the $p$-th power Frobenius map.

In \cite{kedlaya}, Kedlaya showed how $Z(X,T)$ can be determined efficiently, in the case when 
$X$ is a hyperelliptic curve and the characteristic $p$ is odd, by explicitly computing the action 
of $\Frob_p$ on $\Hrig^1(X)$. His algorithm was then extended to characteristic $2$ \cite{dvhyp} and also to
superelliptic curves \cite{gaugu}, $C_{ab}$ curves \cite{dv} and nondegenerate curves \cite{cdv}. 
However, for $C_{ab}$ and nondegenerate curves these algorithms have proved a lot less efficient in practice
than for hyperelliptic and superelliptic curves. The main reason for this is that the algorithms for $C_{a,b}$ 
and nondegenerate curves use a more complicated Frobenius lift that does not send $x$ to $x^p$ anymore. 
Moreover, in the case of nondegenerate curves, the linear algebra that is used 
to compute in the cohomology is not very efficient and when the curve admits a low degree map to
the projective line, as is the case for most nondegenerate curves, this is not fully exploited.

In this paper we propose a new algorithm for computing $Z(X,T)$ that avoids
these problems and can be applied to more general curves as well. Our approach 
combines Kedlaya's original algorithm and Lauder's fibration method \cite{lauder}. 
In the work of Lauder, the Frobenius lift is computed by solving a 
$p$-adic differential equation. For curves it turns out to be more efficient to compute
the Frobenius lift directly by Hensel lifting as in Kedlaya's algorithm, especially since 
this allows one to avoid the radix conversions that take up most of the time in 
the examples of the fibration method computed by Walker in his thesis \cite{walk}.

Our approach can be summarised as follows. We start with a finite separable map $x$
from the curve $X$ to the projective line. After removing the ramification locus of $x$ from the curve, we can 
choose a Frobenius lift that sends $x$ to $x^p$, which we compute by Hensel lifting as in Kedlaya's 
algorithm. We then compute in the cohomology as in Lauder's fibration method to find the matrix of
Frobenius and the zeta function of $X$. 

Let $x:X \rightarrow \mathbf{P}^1_{\FF_q}$ be a finite separable map of degree $d_x$
and $y:X \rightarrow \mathbf{P}^1_{\FF_q}$ a rational function that generates the 
function field of $X$ over $\FF_q(x)$, such that $Q(x,y)=0$ where $Q \in \FF_q[x,y]$ 
is irreducible and monic in $y$ (of degree $d_x$). The polynomial $Q$ is the natural input to our
algorithm. The degree of $Q$ in $x$ will be denoted by $d_y$. The time complexity
of the algorithm is then $\SoftOh(p d_x^6 d_y^4 n^3)$ by Theorem~\ref{thm:time} and the space complexity
$\SoftOh(p d_x^4 d_y^3 n^3)$ by Theorem~\ref{thm:space}.  

When $Q$ is nondegenerate with respect to its Newton 
polygon $\Gamma$, which is common in the sense of \cite[\S 7.2]{cdv}, we have that $d_x d_y \in \BigOh(g)$. 
The time and space complexity of the algorithm are
then $\SoftOh(pg^6n^3)$ and $\SoftOh(pg^4n^3)$, respectively. Note that this slightly
improves the complexity estimate from \cite{cdv}. Now if additionally we fix $d_x$,
then $d_y \in \BigOh(g)$, so that the time and space complexities of the algorithm are 
$\SoftOh(pg^4n^3)$ and $\SoftOh(pg^3n^3)$, respectively. This extends the complexity 
estimate from \cite{kedlaya} from the case where $d_x=2$ to the case where $d_x$ is only fixed.

Note that the time and space complexities of our algorithm are quasilinear in $p$ and hence not polynomial in the
size of the input which is $\log(p) d_x d_y n$. This is also the case for Kedlaya's 
algorithm and the algorithm from \cite{cdv}. However, for hyperelliptic curves, the dependence on $p$ of the 
time and space complexities of Kedlaya's algorithm has been improved to $\SoftOh(p^{1/2})$ \cite{harvey1} and average 
polynomial time \cite{harvey2} by Harvey. It is an interesting problem whether these ideas can be used to
improve the dependence on $p$ of the complexity of our algorithm as well.

We need some assumptions for the algorithm to work. First, we assume that we have a lift 
$\mathcal{Q} \in \ZZ_q[x,y]$ of the polynomial $Q$ such that 
Assumption~\ref{assump:goodlift} below is satisfied. This basically means that over $\QQ_q$ the branch points 
of the map $x$  and the points lying over it are all distinct modulo $p$.
Second, we assume that the zero locus of $\mathcal{Q}$ in the affine plane with 
coordinates $x,y$ is smooth. The first of these assumptions is essential, but the second one can probably be removed, as
sketched in Section~\ref{sec:complete}. Finally, we suppose that we can compute certain integral bases in function fields 
and exclude the time and space required to do so from our complexity estimates. 

We have written a publicly available implementation of our algorithm in the computer algebra
package Magma~\cite{magma}. This implementation turns out to be quite practical and seems to work 
for almost all polynomials $Q$ as illustrated by the example files that come with the code. This should be 
contrasted with the algorithm from \cite{cdv}, which was never fully implemented because it was expected 
not to be practical. Indeed, in some special cases where we have compared our new algorithm against our experimental 
implementation of the algorithm from \cite{cdv}, the new algorithm runs faster by at least two orders of magnitude.

The author was supported by FWO-Vlaanderen. We thank the referees for their useful
comments and suggestions.

\section{Lifting the curve and Frobenius}

\label{sec:lift}

Recall that $X$ is a smooth proper algebraic curve of genus $g$ over the finite field $\FF_q$ of characteristic $p$ and cardinality $q=p^n$. 
Let $x:X \rightarrow \mathbf{P}^1_{\FF_q}$ be a finite separable map 
of degree $d_x$ and $y:X \rightarrow \mathbf{P}^1_{\FF_q}$ a rational function that generates 
the function field of $X$ over $\FF_q(x)$, such that $Q(x,y)=0$ where $Q \in \FF_q[x,y]$ 
is irreducible and monic in $y$ (of degree $d_x$). The degree of $Q$ 
in $x$ will be denoted by $d_y$. Let $\mathcal{Q} \in \ZZ_q[x,y]$ be a 
lift of $Q$ that contains the same monomials in its support as 
$Q$ and is still monic in $y$.

\begin{prop}
The ring $\mathcal{A} = \ZZ_q[x,y]/(\mathcal{Q})$ is a free module of rank $d_x$ over $\ZZ_q[x]$ and a basis is 
given by $[1,y,\dotsc,y^{d_x-1}]$.
\end{prop}

\begin{proof}
This follows from the fact that $\mathcal{Q}$ is monic in $y$.
\end{proof}

\begin{defn}
We let $\Delta(x) \in \ZZ_q[x]$ denote the discriminant of 
$\mathcal{Q}$ with respect 
to the variable $y$ and $r(x) \in \ZZ_q[x]$ the squarefree polynomial
$r=\Delta/(\gcd(\Delta,\frac{d\Delta}{dx}))$. Note that $\Delta(x) \neq 0 \pmod{p}$
since $x:X \rightarrow \mathbf{P}^1_{\FF_q}$ is separable. We denote
\begin{align*}
\mathcal{S}&= \ZZ_q\left[x,1/r \right], &\mathcal{R}&= \ZZ_q\left[x,1/r,y\right]/(\mathcal{Q}),
\end{align*}
and write $\mathcal{V}= \Spec \mathcal{S}$, $\mathcal{U}= \Spec \mathcal{R}$,
so that $x$ defines a finite \'etale morphism from $\mathcal{U}$ to $\mathcal{V}$. 
Finally, we let $U=\mathcal{U} \otimes_{\ZZ_q} \FF_q$, $V=\mathcal{V} \otimes_{\ZZ_q} \FF_q$ 
denote the special fibres and $\mathbb{U}=\mathcal{U} \otimes_{\ZZ_q} \QQ_q$, 
$\mathbb{V}=\mathcal{V} \otimes_{\ZZ_q} \QQ_q$ the generic fibres of
$\mathcal{U}$ and $\mathcal{V}$, respectively.
\end{defn}

\begin{assump}\label{assump:goodlift} We will assume that:
\begin{enumerate} 
\item There exists a smooth proper curve $\mathcal{X}$ over $\ZZ_q$ and a smooth relative 
divisor $\mathcal{D}_{\mathcal{X}}$ on $\mathcal{X}$ such that $\mathcal{U} = \mathcal{X} \setminus \mathcal{D}_{\mathcal{X}}$. 
\item There exists a smooth relative divisor $\mathcal{D}_{\mathbf{P}^1}$ on $\mathbf{P}^1_{\ZZ_q}$ 
such that $\mathcal{V} = \mathbf{P}^1_{\ZZ_q} \setminus \mathcal{D}_{\mathbf{P}^1}$.  
\end{enumerate}
\end{assump}

We write $\mathbb{X}=\mathcal{X} \otimes \QQ_q$ for the generic fibre of $\mathcal{X}$.

\begin{rem}
A relative divisor $\mathcal{D}$ on a smooth curve over $\ZZ_q$ is smooth over $\ZZ_q$ if and only if it is 
reduced and all of the points in its support are smooth over $\ZZ_q$, or equivalently if and only if it reduces 
modulo~$p$ to a reduced divisor $D$. Hence by Assumption~\ref{assump:goodlift}, all branch points of the 
map $x$ restricted to $\mathbb{X}$, and all points on $\mathbb{X}$ lying over these branch points, are distinct modulo~$p$. 
\end{rem}

At every point $P \in \mathcal{X} \setminus \mathcal{U}$, 
we let $z_P$ denote an \'etale local coordinate on $\mathcal{X}$.
By a slight abuse of notation, we write $\ord_P(\cdot)$ for the discrete valuation on $\mathcal{O}_{\mathbb{X},P}$.
We let $e_P$ denote the ramification index of the map $x$. Note that 
the $e_P$ are the same on $X$ as on $\mathbb{X}$, since they can only increase under reduction modulo~$p$, but add 
up to $d_x$ in every fibre.

\begin{assump}\label{assump:Qsmooth}
We will assume that the zero locus of $\mathcal{Q}(x,y)$ in $\mathbf{A}^2_{\QQ_q}$ is
smooth.
\end{assump}

\begin{prop} \label{prop:s}
The element
\[
s(x,y)=r(x)/\frac{\partial \mathcal{Q}}{\partial y}
\]
of $\QQ_q(x,y)$ is contained in $\mathcal{A}$.
\begin{proof}
For $k \in \NN$, we let $W_k$ denote the free $\ZZ_q[x]$-module of polynomials in $\ZZ_q[x,y]$ of degree
at most $k-1$ in the variable $y$. Let $\Sigma$ be the matrix of the $\ZZ_q[x]$-module homomorphism:
\begin{align} \label{eq:alphabeta}
W_{d-1} \oplus W_{d} &\rightarrow W_{2d-1}, &(a,b) \mapsto a \mathcal{Q} + b \frac{\partial \mathcal{Q}}{\partial y},
\end{align}
with respect to the bases $[1,y,\dotsc,y^{d_x-2}]$, $[1,y,\dotsc, y^{d_x-1}]$ and $[1,y,\dotsc,y^{2d_x-2}]$. 
By definition we have $\Delta = \det(\Sigma)$, so that $\Delta$ is contained in the image of~\eqref{eq:alphabeta} and
$\Delta(x)/\frac{\partial \mathcal{Q}}{\partial y}$ is contained in $\mathcal{A}$. By Assumption~\ref{assump:Qsmooth}, 
the ring $\mathcal{A} \otimes \QQ_q$ is the integral closure of $\QQ_q[x]$ in $\QQ_q(x,y)$. Note that the basis 
$[1, y, \dotsc, y^{d_x-1}]$ of $\mathcal{A} \otimes \QQ_q$ is therefore an integral basis for $\QQ_q(x,y)$ over $\QQ_q[x]$.
Since $\mathcal{Q}$ is monic in $y$, for any irreducible polynomial $\pi \in \QQ_q[x]$ the element $\frac{\partial \mathcal{Q}}{\partial y}/\pi$ 
of $\QQ_q(x,y)$ is not integral at the place $(\pi)$, and hence its inverse $\pi/\frac{\partial \mathcal{Q}}{\partial y}$ is integral 
(even zero) at $(\pi)$. Hence $s$ is contained in $\mathcal{A}$.
\end{proof}
\end{prop}

\begin{defn} We denote the ring of overconvergent functions on $\mathcal{U}$ by
\[
\mathcal{R}^{\dag} = \ZZ_q \langle x, 1/r, y \rangle^{\dag}/(\mathcal{Q}). 
\]
Note that $\mathcal{R}^{\dag}$ is a free module of rank $d_x$ over $\mathcal{S}^{\dag}=\ZZ_q \langle x, 1/r \rangle^{\dag}$ 
and that a basis is given by $[y^0, \dotsc, y^{d_x-1}]$. A Frobenius lift 
$\Frob_p:\mathcal{R}^{\dag} \rightarrow \mathcal{R}^{\dag}$ is defined as a $\sigma$-semilinear 
ring homomorphism that reduces modulo $p$ to the $p$-th power Frobenius map. 
 \end{defn}

\begin{thm} \label{thm:froblift} There exists a Frobenius lift $\Frob_p: \mathcal{R}^{\dag} \rightarrow \mathcal{R}^{\dag}$ for which
$\Frob_p(x)=x^p$. 
\end{thm}

\begin{proof}
Define sequences $(\alpha_i)_{i \geq 0}$, $(\beta_i)_{i \geq 0}$, 
with $\alpha_i \in S^{\dag}$ and $\beta_i \in \mathcal{R}^{\dag}$, 
by the following recursion:
\begin{align*}
&\alpha_0=\frac{1}{r^{p}}, \\
&\beta_0=y^p, \\
&\alpha_{i+1}=\alpha_i(2-\alpha_i r^{\sigma}(x^p)) &\pmod{ \; p^{2^{i+1}}}, \\ 
&\beta_{i+1}= \beta_i - \mathcal{Q}^{\sigma}(x^p,\beta_i) s^{\sigma}(x^p, \beta_i) \alpha_i &\pmod{ \; p^{2^{i+1}}}.
\end{align*}
Then one easily checks that the $\sigma$-semilinear ringhomomorphism $\Frob_p: \mathcal{R}^{\dag} \rightarrow \mathcal{R}^{\dag}$ defined by
\begin{align*}
\Frob_p\bigl(x\bigr)&=x^p, &\Frob_p\left(1/r\right)&=\lim_{i \rightarrow \infty} \alpha_i,
&\Frob_p\bigl(y\bigr)&=\lim_{i \rightarrow \infty} \beta_i,
\end{align*}
is a Frobenius lift.
\end{proof}

\begin{prop} \label{prop:con0}
Let $G \in M_{d_x \times d_x}(\ZZ_q[x,1/r])$ denote the matrix such that
\[
d\left(y^{j}\right) = \sum_{i=0}^{d_x-1} G_{i+1,j+1} y^{i} dx, 
\]
for all $0 \leq j \leq d_x-1$. Then we can write $G=M/r$
with $M \in M_{d_x \times d_x}(\ZZ_q[x])$.
\end{prop}

\begin{proof} This follows from the formula
\begin{align} \label{eq:dyj}
d\left(y^{j}\right) = -j y^{(j-1)} \left( \frac{s}{r} \right) \frac{\partial \mathcal{Q}}{\partial x} dx.
\end{align}
\end{proof}

In the terminology of the fibration method, $G dx$ is the matrix of the Gauss--Manin connection $\nabla$ on the $0$-th higher direct image $\RR^0 x_{*}(\mathcal{O}_{\mathbb{U}})$
with respect to the basis $[1,y,\dotsc,y^{d_x-1}]$. By Proposition~\ref{prop:con0}, this matrix has at most a simple pole at all points 
$\neq \infty$ in the support of $\mathcal{D}_{\mathbf{P}^1}$. At $x=\infty$ we will have to make a change of basis for this to be the case.

\begin{assump} \label{assump:infty}
We will assume that a matrix $W^{\infty} \in Gl_{d_x}(\ZZ_q[x,x^{-1}])$ 
is known such that if we denote
$b^{\infty}_j = \sum_{i=0}^{d_x-1} W^{\infty}_{i+1, j+1} y^i$
for all $0 \leq j \leq d_x-1$, then $[b_0^{\infty}, \dotsc, b_{d_x-1}^{\infty}]$
is an integral basis for $\QQ_q(x,y)$ over $\QQ_q[x^{-1}]$.
\end{assump}

\begin{prop} \label{prop:coninf}
Let $G^{\infty} \in M_{d_x \times d_x}(\ZZ_q[x,x^{-1},1/r])$ denote the matrix such that
\[
db_j^{\infty} = \sum_{i=0}^{d_x-1} G_{i+1,j+1}^{\infty} b_i^{\infty} dx,
\]
for all $0 \leq j \leq d_x-1$. Then
$G^{\infty} dx$ has at most a simple pole at $x=\infty$.
\end{prop}

\begin{proof}
We denote $t=1/x$ and let $H \in M_{d_x \times d_x}(\QQ_q(t))$ be defined by $H(t) dt = G^{\infty}(x) dx$. 
Note that $\ord_P(dt/t) = -1$ at every point $P \in \mathcal{X} \setminus \mathcal{U}$ lying over $t=0$. 
At every such $P$ and for all $0 \leq i \leq d_x-1$ we clearly have $\ord_P(db_i^{\infty}) \geq 0$, so that 
$\ord_P(t db_i^{\infty}) - \ord_P(dt) \geq 1$. 
Since $[b_0^{\infty}, \dotsc, b_{d_x-1}^{\infty}]$ is an integral basis for $\QQ_q(x,y)$ over $\QQ_q[t]$, we conclude that $tH$ does
not have a pole at $t=0$, so that $H dt$ has at most a simple pole there.
\end{proof}

\begin{defn}
Let $x_0 \neq \infty$ be a geometric point of $\PP^1(\bar{\QQ}_q)$. The exponents of $G dx$ at $x_0$ are defined as the
eigenvalues of the residue matrix $(x-x_0)G|_{x=x_0}$. Moreover, the exponents of $G^{\infty} dx$ at $x=\infty$ are defined as 
its exponents at $t=0$, after substituting $x=1/t$.
\end{defn}

\begin{prop} \label{prop:exps} The exponents of $G dx$ at any point $x_0 \neq \infty$ and the exponents of $G^{\infty} dx$ at
$x = \infty$ are elements of $\QQ \cap \ZZ_p$ and are contained in the interval $[0,1)$. 
\end{prop}

\begin{proof}
Let $\lambda \in \bar{\QQ}_q$ denote an exponent of $Gdx$ at $x_0 \neq \infty$.
Then there exists $f=\sum_{i=0}^{d_x-1} a_i y^i$ with $a_0,\dotsc,a_{d_x-1} \in \bar{\QQ}_q$ such that  
\begin{equation}\label{eq:res}
df = \left( \frac{\lambda f}{x-x_0} + g \right) dx
\end{equation}
as $1$-forms on $\mathbb{U} \otimes \bar{\QQ}_q$, where $g \in \mathcal{O}(\mathbb{U} \otimes \bar{\QQ}_q)$ satisfies $\ord_P(g) \geq 0$
at all points $P \in x^{-1}(x_0)$. Note that for at least one $P \in x^{-1}(x_0)$ we have $\ord_P(f) < \ord_P(x-x_0)$, 
since otherwise $f/(x-x_0)$ would be integral over $\QQ_q[x]$, contradicting Assumption~\ref{assump:Qsmooth}. For such a $P$, 
dividing by $f$ in \eqref{eq:res} and taking residues, we obtain
\[
\ord_P(f)=\lambda \ord_P(x-x_0) = \lambda e_P.
\]
Since $0 \leq \ord_P(f) < \ord_P(x-x_0)$, we see that $\lambda \in \QQ \cap [0,1)$. 
By Assumption~\ref{assump:goodlift}, elements of $\mathcal{S}$ have $p$-adically integral Laurent series expansions at $x_0$, so
that $(x-x_0)G|_{x=x_0} \in M_{d_x \times d_x}(\ZZ_q)$. Since $p$-adically integral matrices have $p$-adically integral eigenvalues, 
we conclude that $\lambda \in \ZZ_p$. To obtain the same result for the exponents of $G^{\infty} dx$ at $x=\infty$, replace $x_0$ by $\infty$ and 
$(x-x_0)$ by $t=1/x$ in the argument.
\end{proof}

\begin{defn} For a geometric point $x_0 \in \mathbf{P}^1(\bar{\QQ}_q)$, we let $\ord_{x_0}(\cdot)$ denote the discrete valuation on
$\bar{\QQ}_q(x)$ corresponding to $x_0$. We extend these definitions to matrices over $\bar{\QQ}_q(x)$ by taking 
the minimum over their entries.
\end{defn}

\begin{prop} \label{prop:convbound}
Let $N \in \NN$ be a positive integer. 
\begin{enumerate}
\item The element $\Frob_p(1/r)$ of $\mathcal{S}^{\dag}$ is congruent modulo~$p^N$ to 
\[\sum_{i=p}^{pN} \frac{\rho_i(x)}{r^i},
\]
 where  
$\rho_i \in \ZZ_q[x]$ satisfies $\deg(\rho_i)<\deg(r)$ for all $p \leq i \leq pN$.
\item For all $0 \leq i \leq d_x-1$, the element $\Frob_p(y^i)$ of $\mathcal{R}^{\dag}$ is congruent modulo~$p^N$ to 
$\sum_{j=0}^{d-1} \phi_{i,j}(x) y^j$, where
\[
\phi_{i,j} = \sum_{k=0}^{p(N-1)} \frac{\phi_{i,j,k}(x)}{r^k}
\]
for all $0 \leq j \leq d_x-1$ and $\phi_{i,j,k} \in \ZZ_q[x]$ satisfies
\begin{align*}
\deg(\phi_{i,j,0})& \leq -\ord_{\infty}(W^{\infty})-p \ord_{\infty}((W^{\infty})^{-1}), \\
\deg(\phi_{i,j,k})& <\deg(r),
\end{align*}
for all $0 \leq j \leq d_x-1$ and $1 \leq k \leq p(N-1)$.
\item  For all $0 \leq i \leq d_x-1$, the element $\Frob_p(y^i/r)$ of $\mathcal{R}^{\dag}$ is congruent modulo~$p^N$ 
to  $\sum_{j=0}^{d_x-1} \psi_{i,j}(x) (y^j/r)$, where
\[
\psi_{i,j} = \sum_{k=0}^{pN-1} \frac{\psi_{i,j,k}(x)}{r^k}
\]
for all $0 \leq j \leq d_x-1$ and $\psi_{i,j,k} \in \ZZ_q[x]$ satisfies
\begin{align*}
\deg(\psi_{i,j,0})&\leq -\ord_{\infty}(W^{\infty})-p \ord_{\infty}((W^{\infty})^{-1})-(p-1) \deg(r), \\
\deg(\psi_{i,j,k})&<\deg(r),
\end{align*}
for all $0 \leq j \leq d_x-1$ and $1 \leq k \leq pN-1$.
\end{enumerate}
\end{prop}

\begin{proof} \mbox{ }
\begin{enumerate}
\item Since $r^{\sigma}(x^p) \equiv r^p \pmod{p}$, this follows from
\[
\Frob_p \left(\frac{1}{r} \right)=\frac{1}{r^{\sigma}(x^p)}=\frac{1}{r^p}\left(1-\frac{r^p-r^{\sigma}(x^p)}{r^p}\right)^{-1} 
= \frac{1}{r^p} \sum_{i=0}^{\infty} \left( \frac{r^p-r^{\sigma}(x^p)}{r^p} \right)^i.
\]
\item The matrix $\Phi=(\phi_{i,j}) \in M_{d_x \times d_x}(\mathcal{S}^{\dag})$ defines a $p$-th power Frobenius structure on the
higher direct image
$\RR^0 x_{*}(\mathcal{O}_{\mathbb{U}})$. By definition we have $\ord_p(\Phi) \geq 0$ and by Poincar\'e duality we find that
$\ord_p(\Phi^{-1}) \geq 0$ as well. The result now follows from a theorem of Kedlaya and the author \cite[Corollary 2.6]{kedlayatuitman} 
using Proposition~\ref{prop:exps}.  
\item Analogous to (2). 
\end{enumerate}
\end{proof}

\section{Computing (in) the cohomology}

\label{sec:coho}

\begin{defn} The rigid cohomology of $U$ in degree $1$ can be defined as
\[
\Hrig^1(U) = \fCoKer(d: \mathcal{R}^{\dag} \to \Omega^1(\mathbb{U}) \otimes \mathcal{R}^{\dag}). 
\]
\end{defn}

\begin{thm} \label{thm:comparison}
\begin{align*}
\Hrig^1(U) \cong \HdR^1(\mathbb{U})
\end{align*}
\end{thm}

\begin{proof}
This follows as a special case from the comparison theorem between rigid and de Rham
cohomology of Baldassarri and Chiarellotto \cite{baldachiar}, since by Assumption~\ref{assump:goodlift}
$\mathcal{D}_{\mathcal{X}}$ is smooth over $\ZZ_q$.
\end{proof}

We can effectively reduce any $1$-form to one of low pole order using linear algebra 
following work of Lauder \cite{lauder}. 
The procedure consists of two parts, reducing the pole order at the points not lying over $x=\infty$ 
and at those lying over $x=\infty$, respectively. From now on we let $r'$ denote the polynomial $\frac{dr}{dx}$. 
We start with the points not lying over $x=\infty$.

\begin{prop} \label{prop:finitered}
For all $\ell \in \NN$ and every vector $w \in \QQ_q[x]^{\oplus d_x}$, 
there exist vectors $u,v \in \QQ_q[x]^{\oplus d_x}$
with $\deg(v) < \deg(r)$, such that
\begin{align*}
\frac{\sum_{i=0}^{d_x-1} w_i y^i}{r^{\ell}} \frac{dx}{r} &= d\left(\frac{\sum_{i=0}^{d_x-1} v_i y^i}{r^{\ell}} \right)+\frac{\sum_{i=0}^{d_x-1}u_i y^i}{r^{\ell-1}} \frac{dx}{r}.
\end{align*}
\end{prop}

\begin{proof}
Note that since $r$ is separable, $r'$ is invertible in the ring $\QQ_q[x]/(r)$.
One checks that $v$ has to satisfy the $d_x \times d_x$ linear 
system 
\[
\left(\frac{M}{r'} - \ell I \right) v \equiv \frac{w}{r'} \pmod{r}
\] 
over $\QQ_q[x]/(r)$. However, since $\ell \geq 1$ is not an exponent of $Gdx$ by Proposition~\ref{prop:exps}, 
we have that $\det(\ell I-M/r')$ is invertible in $\QQ_q[x]/(r)$, so that this system 
has a unique solution $v$. We now take
\[
u = \frac{w - \left( M-\ell r'I \right) v}{r} - \frac{dv}{dx}.
\]
\end{proof}

We now move on to the points lying over $x=\infty$.

\begin{prop} \label{prop:infinitered}
For every vector $w \in \QQ_q[x,x^{-1}]^{\oplus d_x}$ with 
\[
\ord_{\infty}(w) \leq - \deg(r),
\] 
there exist vectors $u,v \in \QQ_q[x,x^{-1}]^{\oplus d_x}$ with $\ord_{\infty}(u) > \ord_{\infty}(w)$ such that
\begin{align*}
\left(\sum_{i=0}^{d_x-1} w_i b^{\infty}_i\right) \frac{dx}{r} =  d\left(\sum_{i=0}^{d_x-1}v_i b^{\infty}_i\right)+\left(\sum_{i=0}^{d_x-1}u_i b^{\infty}_i \right) \frac{dx}{r}.
\end{align*}
\end{prop}

\begin{proof} We still denote $t=1/x$. By Proposition~\ref{prop:coninf}, we can expand
\begin{align*}
G^{\infty} dx &= \left( \frac{G^{\infty}_{-1}}{t} + G^{\infty}_0 + \dotsc \right) dt, \\
\intertext{where $G^{\infty}_i \in M_{d_x \times d_x}(\QQ_q)$ for all $i \geq -1$.
Writing $m=-\ord_{\infty}(w)-\deg(r)+1$, we can also expand} 
w \frac{dx}{r} &= \sum_{j=-(m+1)}^{\infty} \bar{w}_j t^j dt,
\end{align*}
where $\bar{w}_j \in \QQ_q^{\oplus d_x}$ for all $j \geq -(m+1)$.
Note that 
$m \geq 1$. By Proposition~\ref{prop:exps}, we have 
that $\det(mI-G^{\infty}_{-1})$ is nonzero, so that the linear system
\[
(G^{\infty}_{-1}-mI) \bar{v} = \bar{w}_{-(m+1)} 
\]
has a unique solution $\bar{v} \in \QQ_q^{\oplus d_x}$. We can now take
\begin{align*}
v &= \bar{v} x^{m}, &u&= w-r\left(G^{\infty} v+\frac{dv}{dx}\right).
\end{align*}
\end{proof}

\begin{rem} \label{rem:ord0W}
Note that when $\ord_{\infty}(w) \leq \ord_0(W^{\infty})-\deg(r)+1$, we have that $\ord_0(v) \geq -\ord_0(W^{\infty})$, so that the 
function $\sum_{i=0}^{d_x-1}v_i b^{\infty}_i$ only has poles at points lying over $x=\infty$.
\end{rem}

We now give an explicit description of the cohomology space $\Hrig^1(U)$.

\begin{thm} \label{thm:cohobasis}
Define the following $\QQ_q$-vector spaces:
\begin{align*}
E_0 &= \Bigg\{ \left( \sum_{i=0}^{d_x-1} u_i(x) y^i \right) \frac{dx}{r} &\colon& u \in \QQ_q[x]^{\oplus d_x} \Bigg \}, \\
E_{\infty}&= \Bigg \{ \left( \sum_{i=0}^{d_x-1} u_i(x,x^{-1}) b_{i}^{\infty} \right) \frac{dx}{r} &\colon& u \in \QQ_q[x,x^{-1}]^{\oplus d_x}, \ord_{\infty}(u) > \ord_0(W^{\infty})-\deg(r)+1 \Bigg \}, \\
B_0 &= \bigg \{ \sum_{i=0}^{d_x-1} v_i(x) y^i &\colon& v \in  \QQ_q[x]^{\oplus d_x} \bigg \}, \\ 
B_{\infty}&=\bigg \{ \sum_{i=0}^{d_x-1} v_i(x,x^{-1}) b^{\infty}_i &\colon& v \in \QQ_q[x,x^{-1}]^{\oplus d_x}, \ord_{\infty}(v) > \ord_0(W^{\infty}) \bigg \}.
\end{align*}
Then $E_0 \cap E_{\infty}$ and $d(B_0 \cap B_{\infty})$ are finite dimensional $\QQ_q$-vector spaces and 
\begin{align*}
\Hrig^1(U) &\cong (E_0 \cap E_{\infty})/d(B_0 \cap B_{\infty}). \\
\end{align*}
\end{thm}
\begin{proof}
First, note that elements of $E_0,B_0$ have bounded poles everywhere but at the points lying over $x=\infty$ and elements of $E_{\infty},B_{\infty}$
everywhere but at the points lying over $x=0$. So elements of $E_0 \cap E_{\infty}$ and $d(B_0 \cap B_{\infty})$
have bounded poles everywhere on $\mathbb{X}$. Hence these vector spaces are contained in the space of global sections of some line bundle on $\mathbb{X}$ and
are therefore finite dimensional.

Next, we show that every class in $\Hrig^1(U)$ can be represented by a $1$-form in $E_0 \cap E_{\infty}$. Note that by Theorem~\ref{thm:comparison}
we can restrict to classes in $\HdR^1(\mathbb{U})$. Now every such class can be represented by a $1$-form in $E_0$ by (repeatedly) applying 
Proposition~\ref{prop:finitered}. Then we change basis by the matrix $W^{\infty}$ from Assumption~\ref{assump:infty}. Observe that this change of basis
might introduce a pole at $x=0$. Now our cohomology class can be represented by $1$-form in $E_0 \cap E_{\infty}$ by (repeatedly) applying 
Proposition~\ref{prop:infinitered} and Remark~\ref{rem:ord0W}. 

Finally, we have to prove that if a $1$-form $\omega \in E_0 \cap E_{\infty}$ is exact, then it lies in $d(B_0 \cap B_{\infty})$. So let
$\omega \in E_0 \cap E_{\infty}$ denote such an exact $1$-form. From Assumption~\ref{assump:Qsmooth} and the definition of
$[b_0^{\infty},\dotsc,b_{d_x-1}^{\infty}]$, it follows that $\ord_P(\omega) \geq -1$ all points $P$ not lying over $x=\infty$ and 
$\ord_P(\omega) \geq \ord_0(W^{\infty}+1)e_P - 1$ at all points $P$ lying over $x=\infty$. 
Note that the exterior derivative lowers the order by at most $1$. So if $\omega=df$ for some $f \in \mathcal{O}(\mathbb{U})$, then
$\ord_P(f) \geq 0$ at all points $P$ not lying over $x=\infty$ and $\ord_P(f) \geq (\ord_0(W^{\infty})+1) e_P$ at all points $P$ lying
over $x=\infty$. Using Assumption~\ref{assump:Qsmooth} and the definition of
$[b_0^{\infty},\dotsc,b_{d_x-1}^{\infty}]$ again, it follows that $f$ is an element of  $B_0 \cap B_{\infty}$.
\end{proof}

Note that by the proof of Theorem~\ref{thm:cohobasis}, we can effectively reduce any $1$-form to one
in $E_0 \cap E_{\infty}$ with the same cohomology class. However, the reduction procedure will introduce 
$p$-adic denominators and therefore suffer from loss of $p$-adic precision. In the following two propositions 
we bound these denominators. Our bounds and their proofs generalise the ones from \cite{kedlaya}. 

\begin{prop} \label{prop:finiteprecision}
Let $\omega \in \Omega^1(\mathcal{U})$ be of the form
\[
\omega=\frac{\sum_{i=0}^{d_x-1} w_i y^i}{r^{\ell}} \frac{dx}{r},
\]
where $\ell \in \NN$ and $w \in \ZZ_q[x]^{\oplus d_x}$ satisfies $\deg(w)<\deg(r)$.
We define 
\[
e = \max \{ e_P | P \in \mathcal{X} \setminus \mathcal{U}, x(P) \neq \infty \}.
\]
If we represent the class of $\omega$ in $\Hrig^1(U)$ by  
\[
\left(\sum_{i=0}^{d_x-1} u_i y^i \right) \frac{dx}{r},
\]
with $u \in \QQ_q[x]^{\oplus d_x}$ as in the proof of Theorem~\ref{thm:cohobasis}, then
\[
p^{\lfloor \log_p(\ell e) \rfloor} u \in \ZZ_q[x]^{\oplus d_x}.
\]
\end{prop}

\begin{proof}
We have
\[
\omega =df+\left(\sum_{i=0}^{d_x-1} u_i y^i \right) \frac{dx}{r} 
\]
with $f = \sum_{j=1}^{\ell} (\sum_{i=0}^{d_x-1} (v_j)_i y^i)/r^j$, where $v_j \in \QQ_q[x]^{\oplus d_x}$ satisfies $\deg(f_j) < \deg(r)$ for all
$1 \leq j \leq \ell$. Note that it is sufficient to show that $p^{\lfloor \log_p(\ell e) \rfloor}f \in \mathcal{R}$. 
By Assumption~\ref{assump:goodlift}, we have that 
\begin{align*}
\mathcal{O}(\mathcal{X}-x^{-1}(\infty))/(r)^k \cong \prod_{P \in \mathcal{X} \setminus \mathcal{U}, x(P) \neq \infty} \mathcal{O}_{\mathcal{X},P}/(z_P^{e_P})^{k},
\end{align*}
for all $k \in \NN$. Moreover, we have that $\mathcal{O}(\mathbb{X}-x^{-1}(\infty)) \cong \mathcal{A} \otimes \QQ_q$ by Assumption~\ref{assump:Qsmooth}.
To show that $p^{\lfloor \log_p(\ell e) \rfloor} f$ is integral, it is therefore enough to show that for every $P \in \mathcal{X} \setminus \mathcal{U}$ 
with $x(P) \neq 0$, the Laurent series expansion 
\[
a_{-\ell e_P} z_P^{-\ell e_P}+ \dotsc + a_{-e_P-1} z_P^{-e_P-1} + \mathcal{O}(z_P^{-e_P})
\]
of $p^{\lfloor \log_p(\ell e) \rfloor} f$ is integral. However, the differential $df$ has a pole of order at most $\ell e_P+1$ at $P$, and its Laurent series expansion 
\[
\Bigl( b_{-\ell e_P-1} z_P^{-\ell e_P-1} + \dotsc + b_{-e_P-2} z_P^{-e_P-2} + \mathcal{O}(z_P^{-e_P-1}) \Bigr) dz_P
\]
is integral since $\omega$ is integral. The worst denominator we get by integrating this series is therefore $p^{\lfloor \log_p(\ell e) \rfloor}$ and the result follows.
\end{proof}

\begin{prop} \label{prop:infiniteprecision}
Let $\omega \in \Omega^1(\mathcal{U})$ be of the form
\[
\omega=(\sum_{i=0}^{d_x-1} w_i(x,x^{-1}) b_i^{\infty}) \frac{dx}{r},
\]
where $w \in \ZZ_q[x,x^{-1}]^{\oplus d_x}$ satisfies $\ord_{\infty}(w) \leq \ord_0(W^{\infty})-\deg(r)+1$. 
We define
\begin{align*}
m               &= -\ord_{\infty}(w)-\deg(r)+1, \\
e_{\infty}      &= \max \{ e_P | P \in \mathcal{X} \setminus \mathcal{U}, x(P) = \infty \}. 
\end{align*}
If we represent the class of $\omega$ in $\Hrig^1(U)$ by  
\[
\left(\sum_{i=0}^{d_x-1} u_i y^i \right) \frac{dx}{r},
\]
with $u \in \QQ_q[x,x^{-1}]^{\oplus d_x}$ such that $\ord_{\infty}(u) > \ord_0(W^{\infty})-\deg(r)+1$ as in the proof of Theorem~\ref{thm:cohobasis}, then
\[
p^{\lfloor \log_p(m e_{\infty}) \rfloor} u \in \ZZ_q[x,x^{-1}]^{\oplus d_x}.
\]
\end{prop}

\begin{proof}
We have
\[
\omega =df+\left(\sum_{i=0}^{d_x-1} u_i y^i \right) \frac{dx}{r} 
\]
with $f = \sum_{j=-\ord_0(W^{\infty})}^{m} (\sum_{i=0}^{d_x-1} (v_j)_i y^i) x^j$, where $v_j \in \QQ_q^{\oplus d_x}$ for all
$-\ord_0(W^{\infty}) \leq j \leq m$. Note that it is sufficient to show that $p^{\lfloor \log_p(\ell e) \rfloor}f \in \mathcal{R}$. 
By Assumption~\ref{assump:goodlift}, we have that 
\begin{align} \label{eqn:decompinf}
\mathcal{O}(\mathcal{X}-x^{-1}(0))/(t)^k \cong \prod_{P \in \mathcal{X} \setminus \mathcal{U}, x(P) = \infty} \mathcal{O}_{\mathcal{X},P}/(z_P^{e_P})^{k},
\end{align}
for all $k \in \NN$. Moreover, by definition $[b_0^{\infty},\dotsc,b_{d_x-1}^{\infty}]$ is a basis for $\mathcal{O}(\mathbb{X}-x^{-1}(0))$ over 
$\QQ_q[x^{-1}]$. To show that $p^{\lfloor \log_p(\ell e_{\infty}) \rfloor} f$ is integral, it is therefore enough to show that for every $P \in \mathcal{X} \setminus \mathcal{U}$ 
with $x(P)=0$, the Laurent series expansion 
\[
a_{-m e_P} z_P^{-m e_P}+ \dotsc + a_{(\ord_0(W^{\infty})+1)e_P-1} z_P^{(\ord_0(W^{\infty})+1)e_P-1} + \mathcal{O}(z_P^{(\ord_0(W^{\infty})+1)e_P})
\]
of $p^{\lfloor \log_p(\ell e_{\infty}) \rfloor} f$ is integral. However, the differential $df$ has a pole of order at most $m e_P+1$ at $P$, and its Laurent series expansion 
\[
\left( b_{-m e_P-1} z_P^{-m e_P-1} + \dotsc + b_{(\ord_0(W^{\infty})+1)e_P} z_P^{(\ord_0(W^{\infty})+1)e_P} + \mathcal{O}\left(z_P^{(\ord_0(W^{\infty})+1)e_P-1}\right) \right) dz_P
\]
is integral since $\omega$ is integral. The worst denominator we get by integrating this series is therefore $p^{\lfloor \log_p(m e_{\infty}) \rfloor}$ and the result follows.
\end{proof}

\begin{rem}
Note that Propositions~\ref{prop:finitered}, \ref{prop:infinitered}, \ref{prop:finiteprecision} and \ref{prop:infiniteprecision}
can be used to give an alternative effective proof of Theorem~\ref{thm:comparison}.
\end{rem}

Recall that in Theorem~\ref{thm:cohobasis} the computation of a basis for
$\Hrig^1(U)$ was reduced to a (small) finite dimensional linear algebra problem. However, the dimension of $\Hrig^1(U)$ is generally 
about $d_x$ times the dimension of $\Hrig^1(X)$, so that we would like to compute a basis for this last space. For this we will need 
to compute the kernel of a cohomological residue map. 

\begin{defn}
For a $1$-form $\omega \in \Omega^1(\mathcal{U})$ and a point $P \in \mathcal{X} \setminus \mathcal{U}$, 
we let 
\[ 
res_P(\omega) \in \mathcal{O}_{\mathcal{X},P}/(z_P)
\]
denote the coefficient $a_{-1}$ in the Laurent series expansion
\[
\omega = (a_{-k} z_P^k + \dotsc + a_{-1} z_P^{-1}+ \cdots) dz_P.
\] 
Moreover, we denote
\begin{align*}
res          &= \bigoplus_{P \in \mathcal{X} \setminus \mathcal{U} \colon x(P) \neq \infty} res_P, 
&res_{\infty}&= \bigoplus_{P \in \mathcal{X} \setminus \mathcal{U} \colon x(P) = \infty} res_P.
\end{align*}
\end{defn}

\begin{thm} We have an exact sequence 
\[
\begin{CD}
0 @>>> \Hrig^1(X) @>>> \Hrig^1(U) @>(res \oplus res_{\infty}) \otimes \QQ_q>> \underset{P \in \mathcal{X} \setminus \mathcal{U}}{\bigoplus} \mathcal{O}_{\mathcal{X},P}/(z_P) \otimes \QQ_q.
\end{CD}
\]
\end{thm}

\begin{proof}
This is well known.
\end{proof}

The kernels of $res$ and $res_{\infty}$ can be computed without having to compute the
Laurent series expansions at all $P \in \mathcal{X} \setminus \mathcal{U}$ using the
following two propositions. We start with the residues at the points not lying over $x=\infty$.

\begin{prop} \label{prop:kerres} Let $\omega \in \Omega^1(\mathbb{U})$ be a $1$-form of the form 
\[
\omega=\left(\sum_{i=0}^{d_x-1}u_i(x) y^i \right) \frac{dx}{r},
\]
with $u \in \QQ_q[x]^{\oplus d_x}$. Then 
\[
res(\omega)=0 \; \; \; \Leftrightarrow \; \; \; \frac{\partial Q}{\partial y} \sum_{i=0}^{d_x-1} u_i y^i=0 \; \; \; \mbox{in} \; \; \; \mathcal{O}(\mathbb{X}-x^{-1}(\infty))/(r).
\] 
\end{prop}

\begin{proof}
Let $P$ run over all points in $\mathcal{X} \setminus \mathcal{U}$ such that $x(P) \neq \infty$. One checks that $\ord_P(\frac{dx}{r})=-1$ and $\ord_P(\omega) \geq -1$.
Hence $res_P(\omega)=0$ if and only if $\ord_P(\sum_{i=0}^{d_x-1}u_i y^i) \geq 1$. However, since $\ord_P(\frac{\partial Q}{\partial y})=e_P-1$
by Assumption~\ref{assump:Qsmooth}, this is the case if and only if $\ord_P(\frac{\partial Q}{\partial y} \sum_{i=0}^{d_x-1} u_i y^i) \geq e_P$. Finally,
we have that $\ord_P(\frac{\partial Q}{\partial y} \sum_{i=0}^{d_x-1} u_i y^i) \geq e_P$ at all $P$ in $\mathcal{X} \setminus \mathcal{U}$ such that $x(P) \neq \infty$
if and only if $\frac{\partial Q}{\partial y} \sum_{i=0}^{d_x-1} u_i y^i$ maps to $0$ in $\mathcal{O}(\mathbb{X}-x^{-1}(\infty))/(r)$.
\end{proof}

We now move on to the residues at the points lying over $x=\infty$.

\begin{prop} \label{prop:kerresinfty}
Let $\omega \in \Omega^1(\mathbb{U})$ be a $1$-form of the form 
\[
\omega=\left(\sum_{i=0}^{d_x-1}u_i(x,x^{-1}) b_i^{\infty} \right) \frac{dx}{r},
\]
where $u \in \QQ_q[x,x^{-1}]^{\oplus d_x}$ satisfies $\ord_{\infty}(u)>-\deg(r)$, and let
$v \in \QQ_q^{\oplus d_x}$ be defined by $v= \left(x^{1-\deg(r)}u \right)\lvert_{x=\infty}$. 
Moreover, let the residue matrix $G^{\infty}_{-1} \in M_{d_x \times d_x}(\QQ_q)$ be defined 
as in the proof of Proposition~\ref{prop:infinitered}, and let $V_\lambda$ denote the 
generalised eigenspace of $G^{\infty}_{-1}$ with eigenvalue $\lambda$, so that $\QQ_q^{\oplus d_x}$ 
decomposes as $\bigoplus V_{\lambda}$. Then
\[
res_{\infty}(\omega)=0  \; \; \; \Leftrightarrow \; \; \; \mbox{the projection of $v$ onto $V_0$ $=0$}.
\]
\end{prop}

\begin{proof}
Let $P$ run over all points in $\mathcal{X} \setminus \mathcal{U}$ such that $x(P)=\infty$. One checks that $\ord_P(\frac{dx}{r}) = -1+(\deg(r)-1)e_P$ and
$\ord_P(\omega) \geq -1$. Since $\ord_P(x)=-e_P$, we have that $res_P(\omega)=0$ if and only if $\ord_P(\sum_{i=0}^{d_x-1}v_i b^{\infty}_i) \geq 1$. 
We still denote
$t=1/x$. Note that $[b_0^{\infty}, \ldots, b_{d_x-1}^{\infty}]$ is a $\QQ_q$-basis for $\mathcal{O}(\mathbb{X}-x^{-1}(0))/(t)$ and that
\begin{align} \label{eqn:decomp}
\mathcal{O}(\mathbb{X}-x^{-1}(0))/(t) \cong \prod_{P \in \mathcal{X} \setminus \mathcal{U}, x(P)=\infty} \mathcal{O}_{\mathbb{X},P}/(z_P^{e_P}).
\end{align}
Under this isomorphism every factor on the right-hand side is an invariant subspace for $G^{\infty}_{-1}$ since $\ord_P(f) \geq e_P$ implies
that $\ord_P(t df/dt) \geq e_P$. 

We know from Proposition~\ref{prop:exps} that the eigenvalues of $G^{\infty}_{-1}$ are elements of
$\QQ \cap \ZZ_p$ contained in the interval $[0,1)$ and that if $f \in \mathcal{O}(\mathbb{X}-x^{-1}(0))/(t)$ is an eigenvector with
eigenvalue $\lambda$ and $\ord_P(f) < e_P$ for some $P$, then we have that $\ord_P(f) = \lambda e_P$. We claim that the eigenvalues of $G^{\infty}_{-1}$
on the factor corresponding to the point $P$ in \eqref{eqn:decomp} are $[0,1/e_P,\ldots,(e_P-1)/e_P]$. In particular they are all different, so that $G^{\infty}_{-1}$ is
diagonalisable. This follows since locally around the point $P$ the map $t$ is the $e_P$-th power map, so the eigenvalues of its monodromy are all the $e_P$-th roots 
of unity, but these eigenvalues of monodromy are of the form $e^{2\pi i \lambda}$ where $\lambda$ runs over the eigenvalues of $G^{\infty}_{-1}$ on
the factor corresponding to the point $P$ in \eqref{eqn:decomp}.

Now, if we decompose $v$ onto a basis of eigenvectors compatible with the decomposition \eqref{eqn:decomp}, then we see that 
$\ord_P (\sum_{i=0}^{d_x-1}v_i b^{\infty}_i) \geq 1$ for all $P$ in $\mathcal{X} \setminus \mathcal{U}$ such that $x(P)=\infty$ 
if and only if the components along the eigenvectors with eigenvalue $0$ all vanish.
\end{proof}

\begin{rem}
For any $\omega \in \Omega^1(\mathbb{U})$ we can first apply Propositions~\ref{prop:finitered} and~\ref{prop:infinitered} to 
represent the class of $\omega$ in $\Hrig^1(U)$ by $1$-forms to which we can apply Propositions~\ref{prop:kerres} and~\ref{prop:kerresinfty}. 
\end{rem}

\section{The complete algorithm and its complexity}

\label{sec:complete}

In this section we describe all the steps in the algorithm and determine bounds for the complexity. Recall that $X$ is
a curve of genus $g$ over a finite field $\FF_q$ with $q=p^n$ and that $d_x$ and $d_y$ denote the degrees of the defining
polynomial $Q$ in the variables $y$ and $x$, respectively.
All computations are carried out to $p$-adic precision $N$ which will be specified later.  We use the
$\SoftOh(-)$ notation that ignores logarithmic factors, i.e. $\SoftOh(f)$ denotes the class of functions that
lie in $\BigOh(f \log^k(f))$ for some $k \in \NN$. For example, two elements of $\ZZ_q$ can be multiplied 
in time $\SoftOh(\log(p)nN)$. We let $\theta$ denote an exponent for matrix multiplication, so that two 
$k \times k$ matrices can be multiplied in $\BigOh(k^{\theta})$ ring operations.
It is known that $\theta \geq 2$ and that one can take $\theta \leq 2.3729$ \cite{williams2012}. We start with some bounds that will 
be useful later on.

\begin{prop} \label{prop:degr} Let $\Delta$, $s$, $r$ be defined as in Section~\ref{sec:lift} and $e, e_{\infty}$
as in Section~\ref{sec:coho}. We have:
\begin{subequations}
\begin{alignat}{3}
\deg(\Delta), \deg(r), \deg(s) & \leq 2(d_x-1) d_y   &\; \in \;& \BigOh(d_x d_y), \label{eq:bound1} \\
e, e_{\infty}                  & \leq  d_x              &\; \in \;& \BigOh(d_x), \label{eq:bound3} \\
g                              & \leq (d_x-1)(d_y-1) &\; \in \;& \BigOh(d_x d_y). \label{eq:bound4}
\end{alignat}
\end{subequations}
\end{prop}

\begin{proof} 
\eqref{eq:bound1} Note that the matrix $\Sigma$ from Proposition~\ref{prop:s} 
is a $(2d_x-1) \times (2d_x-1)$ matrix over $\ZZ_q[x]$ of degree at most $d_y$ and that the row
corresponding to $y^{2d_x-2}$ has degree $0$. Since $\Delta=\det(\Sigma)$, this implies
that $\deg(\Delta) \leq (2d_x-2)d_y$. Writing $s = \sum_{i=0}^{d_x-1} s_i(x) y^i$
with $s_i \in \ZZ_q[x]$, the $s_i$ are in fact entries of $r \Sigma^{-1}$, so that 
$\deg(s_i) \leq (2d_x-2) d_y$ for all $0 \leq i \leq d_x-1$. \\
\eqref{eq:bound3} All the ramification indices $e_P$ are at most $d_x$.\\
\eqref{eq:bound4} It is known \cite{beelenpellikaan} that $g$ is at most the number of interior points of the Newton polygon 
of $Q$, which is clearly bounded by $(d_x-1)(d_y-1)$.
\end{proof} 

\begin{prop} We have \label{prop:ordbounds}
\begin{subequations}
\begin{alignat}{3}
\ord_{\infty}(W^{\infty})		    	& \geq 	-(d_x-1)d_x d_y   	&\; \in& -\BigOh(d_x^2 d_y),  \\
\ord_{\infty}((W^{\infty})^{-1})		& \geq  -(d_x-1)d_y    		&\; \in& -\BigOh(d_x d_y).     \\
\intertext{Moreover, we may assume that}
\ord_0(W^\infty)                    		& \geq  -(d_x-1)d_y    		&\; \in& -\BigOh(d_x d_y).
\end{alignat}
\end{subequations}
\end{prop}

\begin{proof}
We still denote $t=1/x$. One easily checks that the minimal polynomial $\mathcal{Q}^{\infty}$ of $y'=y/x^{d_y}$ over $\QQ_q[t]$ is monic. 
Hence the functions $1,y', \dotsc, y'^{d_x-1}$ are $\QQ_q[t]$-linear combinations of $b_0^{\infty}, \dotsc, b_{d_x-1}^{\infty}$, so that
$\ord_{\infty}((W^{\infty})^{-1}) \geq -(d_x-1) d_y$. 

Since the degree of $\mathcal{Q}^{\infty}$ in the variable $t$ is at most $d_x d_y$,
its discriminant $\Delta^{\infty} \in \ZZ_q[t]$ with respect to the variable $y'$ has degree $\leq 2(d_x-1)d_x d_y$ by the argument from 
Proposition~\ref{prop:degr}. Defining the matrix $W^{\infty '} \in Gl_{d_x}(\ZZ_q[x,x^{-1}])$ such that
\[
b^{\infty}_j = \sum_{i=0}^{d_x-1} W^{\infty'}_{i+1, j+1} y'^i
\]
for all $0 \leq j \leq d_x-1$, it follows from basic properties of the discriminant that 
$\ord_{\infty}(W^{\infty'}) \geq -\deg(\Delta^{\infty})/2$. Clearly
$\ord_{\infty}(W^{\infty}) \geq \ord_{\infty}(W^{\infty'})$, so this implies that
$\ord_{\infty}(W^{\infty}) \geq -(d_x-1)d_x d_y$.

We may assume that $\ord_0(W^{\infty'}) \geq 0$. When this is not the
case, we can proceed as in \cite{vanhoeij} to obtain another integral basis such that
$\ord_0(W^{\infty'}) \geq 0$. Note that this does not involve computing Puiseux 
expansions etc. as in \cite{vanhoeij}, since we already have the integral basis 
$[b_0^{\infty}, \dotsc, b_{d_x-1}^{\infty}]$ at our disposal. Finally, clearly $\ord_0(W^{\infty'}) \geq 0$
implies that $\ord_0(W^\infty)\geq -(d_x-1)d_y$.
\end{proof}

In general algorithms like the one from \cite{vanhoeij} are available for computing integral bases in function 
fields. In the following important special case we can write down $[b_0^{\infty},\dotsc,b_{d_x-1}^{\infty}]$ 
directly.

\begin{prop} \label{prop:triangle}
For positive integers $a,b \in \mathbb{N}$, let $\Gamma$ denote the triangle in the plane with vertices 
$(0,0),(a,0)$ and $(0,b)$. If $Q$ is nondegenerate with respect to $\Gamma$, then we can take 
$[b_0^{\infty},\dotsc,b_{d_x-1}^{\infty}]$ to be 
\[
\left[1,x^{\lfloor -a/b \rfloor} y, x^{\lfloor -2a/b \rfloor} y^2,\ldots,x^{\lfloor -(b-1)(a/b) \rfloor} y^{b-1} \right].
\]
\end{prop}

\begin{proof}
Let $\Gamma'$ be the translation of $\Gamma$ defined by $\Gamma'=\Gamma-(a,0)$. If $Q$ is nondegenerate with
respect to $\Gamma$, then so is $\mathcal{Q}$ by \cite[Corollary 6]{cdv}. The toric surface $Y_{\Gamma}$ associated
to $\Gamma$ contains $3$ divisors at infinity, corresponding to the edges of $\Gamma$. Now, a Laurent polynomial 
is regular on $\mathbb{X}-x^{-1}(0)$ if and only if is regular at the points
lying on the intersection of $\mathbb{X}$ with the divisors at infinity of $Y_{\Gamma}$ corresponding to the 
horizontal and the diagonal edges of $\Gamma$. Therefore, it follows from (the proof of) \cite[Lemma 2]{cdv}
that
\[
\mathcal{O}(\mathbb{X}-x^{-1}(0)) \cong \QQ_q[\Gamma']/(x^{-a}\mathcal{Q}), 
\]
where $\QQ_q[\Gamma']$ is the algebra over $\QQ_q$ generated by the monomials supported on the cone generated
by $\Gamma'$. Note that we can substract a multiple of $x^{-a} \mathcal{Q}$ from an arbitrary element of $\QQ_q[\Gamma']$ to 
eliminate all powers of $y$ greater than $b$. Therefore, $\mathcal{O}(\mathbb{X}-x^{-1}(0))$ 
is generated as a $\QQ_q[x^{-1}]$-module by the set 
\[\left\{1,x^{\lfloor -a/b \rfloor} y, x^{\lfloor -2a/b \rfloor} y^2,\ldots,x^{\lfloor -(b-1)(a/b) \rfloor} y^{b-1} \right\}.\]
Since the rank of $\mathcal{O}(\mathbb{X}-x^{-1}(0))$ over $\QQ_q[x^{-1}]$ is $b$, this finishes the proof.
\end{proof}

\begin{assump} \label{assump:Winfty}
In the complexity analysis we will assume a couple of times (with explicit mention) that $\ord_{\infty}(W^{\infty}) \in -\BigOh(d_x d_y)$.
\end{assump}

\subsection{Step I: Determine a basis for the cohomology} \mbox{ } \\

We want to find $\omega_1,\dotsc,\omega_{\kappa} \in (E_0 \cap E_{\infty}) \cap \Omega^1(\mathcal{U})$ such that:
\begin{enumerate}
\item $[\omega_1,\dotsc,\omega_{\kappa}]$ is a basis for $\Hrig^1(U) \cong (E_0 \cap E_{\infty})/d(B_0 \cap B_{\infty})$,
\item the class of every element of $(E_0 \cap E_{\infty}) \cap \Omega^1(\mathcal{U})$ in $\Hrig^1(U)$ has $p$-adically integral 
      coordinates with respect to $[\omega_1,\dotsc,\omega_{\kappa}]$,
\item $[\omega_1,\dotsc,\omega_{2g}]$ is a basis for the kernel of $res \oplus res_{\infty}$ and hence for the subspace $\Hrig^1(X)$
      of $\Hrig^1(U)$.
\end{enumerate} 
This can be done using standard linear algebra over $\ZZ_q$, i.e. by computing the Smith normal forms (including
unimodular transformations) of two matrices. Note for an element
\[
\left( \sum_{i=0}^{d_x-1} u_i(x) y^i \right) \frac{dx}{r} \in E_0 \cap E_{\infty},
\]
we have that
$\deg(u) \leq \deg(r)-2-\ord_0(W^{\infty})-\ord_{\infty}(W^{\infty})$. Hence the dimensions of the matrices involved are at most
\[
d_x \left(\deg(r)-1-\ord_0(W^{\infty})-\ord_{\infty}(W^{\infty}) \right).
\]
Therefore, (under Assumption~\ref{assump:Winfty}) we need $\BigOh((d_x^{2} d_y)^{\theta})$ ring operations in $\ZZ_q$ by \cite[Chapter 7]{storjohann}, 
each of which can be carried out in time $\SoftOh(\log(p)nN)$, so that the time complexity of this step is 
\[
\SoftOh \left(\log(p) d_x^{2\theta} d_y^{\theta} n N \right).
\] 

\subsection{Step II: Compute the map $\Frob_p$} \mbox{ } \\

We use Theorem~\ref{thm:froblift} to compute approximations:
\begin{align*}
\Frob_p(1/r) &= \alpha_i+ \mathcal{O}(p^{2^i}), \\
\Frob_p(y)           &= \beta_i+ \mathcal{O}(p^{2^i}),
\end{align*}
for $i=1, \dotsc, \nu=\lceil \log_2(N) \rceil$. 
We carry out all computations using $r$-adic expansions for the elements of $\mathcal{R}$ and $\mathcal{S}$, e.g. we represent $\alpha_i,\beta_i$ as:
\begin{align*}
\alpha_i  &= \sum_{j \in J} \frac{\alpha_{i,j}(x)}{r^j},
&\beta_i  &= \sum_{k=0}^{d_x-1} \Bigl( \sum_{j \in J}  \frac{\beta_{i,j,k}(x)}{r^j} \Bigr) y^k,
\end{align*}
where $J \subset \ZZ$ is finite and $\alpha_{i,j}, \beta_{i,j,k} \in \ZZ_q[x]$ satisfy $\deg(\alpha_{i,j}), \deg(\beta_{i,j,k}) < \deg(r)$,
for all $i,j,k$. By Propositions~\ref{prop:convbound} and~\ref{prop:ordbounds}, we have that
\[
\lvert \min{J} \rvert, \lvert \max{J} \rvert \in \BigOh\Bigl(p\Bigl(N+d_x^2 d_y/\deg(r)\Bigr)\Bigr).
\]
Hence, a single ring operation in $\mathcal{R}$ takes time
\[
\SoftOh(\log(p) \lvert \max{J} - \min{J} \rvert nN) \subset \SoftOh \Bigl(p d_x^2 d_y  \bigl(N+d_x \bigr) n N \Bigr).
\]
Moreover, the image of an element of $\QQ_q$ under the map $\sigma$ can be computed
in time $\SoftOh(\log^2(p) n + \log(p) nN)$ by \cite{hubrechts}.
We need $\BigOh(d_x \log(N))$ ring operations in $\mathcal{R}$ and $\BigOh(d_x d_y)$ 
applications of $\sigma$ in order to compute $(\alpha_{\nu},\beta_{\nu})$. Therefore, this can be 
done in time
\[
\SoftOh\Bigl(p d_x^3 d_y  \bigl(N+d_x \bigr) n N\Bigr).
\]
Now for each 
$\omega_i=(\sum_{k=0}^{d-1} u_k(x) y^k)\frac{dx}{r}$ with $1 \leq i \leq 2g$, 
we compute
\begin{align} \label{eq:Fp1}
\Frob_p (\omega_i)
=\sum_{k=0}^{d_x-1} p x^{p-1} u_k^{\sigma}(x^p) \Frob_p\Bigl(\frac{y^k}{r}\Bigr) dx
=\sum_{k=0}^{d_x-1} p x^{p-1} u_k^{\sigma}(x^p) \alpha_{\nu} \beta_{\nu}^k dx + \BigOh(p^N).
\end{align}
For a single $\omega_i$ this takes $\BigOh(d_x)$ ring operations in $\mathcal{R}$ and $\BigOh(d_x \deg(r))$ applications
of $\sigma$. Hence the complete set of $\Frob_p(\omega_i)$ can be computed 
in time
\[
\SoftOh\Bigl(g p d_x^3 d_y  \bigl(N+d_x \bigr) n N\Bigr) \subset \SoftOh\Bigl(p d_x^4 d_y^2  \bigl(N+d_x \bigr) n N\Bigr),
\]
which is also the total time complexity of this step.

\subsection{Step III: Reduce back to the basis} \mbox{ } \\

We want to find the matrix $\Phi \in M_{2g \times 2g}(\QQ_q)$ such that
\[
\Frob_p(\omega_i) = \sum_{j=1}^{2g} \Phi_{j,i} \omega_j
\]
in $\Hrig^1(U)$. In the previous step, we have obtained an approximation
\begin{equation} \label{eq:Fpomega}
\Frob_p(\omega_i) = \sum_{j \in J} \Bigl( \sum_{k=0}^{d_x-1} \frac{w_{i,j,k}(x)}{r^j} y^k \Bigr) \frac{dx}{r} +\BigOh(p^N),
\end{equation}
where $J \subset \ZZ$ is finite and $w_{i,j,k}(x) \in \ZZ_q[x]$ satisfies
$\deg(w_{i,j,k}(x))< \deg(r)$ for all $i,j,k$. We now use Proposition~\ref{prop:finitered} and Proposition~\ref{prop:infinitered}
(repeatedly) to reduce this $1$-form to an element of $E_0 \cap E_{\infty}$ as in Theorem~\ref{thm:cohobasis}. 

To carry out the reduction procedure, it is sufficient to solve 
a linear system with parameter ($\ell$ or $m$, respectively) only once in 
Propositions~\ref{prop:finitered} and~\ref{prop:infinitered}.
After that, every reduction step corresponds to a multiplication of a vector by a $d_x \times d_x$ matrix 
(over $\QQ_q[x]/(r)$ or $\QQ_q$, respectively). First, the linear systems with parameter can be solved in time 
\[
\SoftOh(\log(p) d_x^{\theta+1} \deg(r) nN) \subset \SoftOh(\log(p) d_x^{\theta+2} d_y nN),
\]
where one factor $d_x$ is from the degree in the parameter.
Then, the number of reduction steps at the points not lying over $x=\infty$ is $\BigOh(pN)$ for each $\Frob_p(\omega_i)$. Every 
single finite reduction step takes time $\SoftOh(\log(p) d_x^2 \deg(r) nN)$, so all $\Frob_p(\omega_i)$ can be reduced in time 
\[
\SoftOh(g (pN) d_x^2 \log(p) \deg(r) nN) \subset \SoftOh(pd_x^4 d_y^2 nN^2).
\]
Finally, the number of reduction steps at the points lying over $x=\infty$ is $\BigOh(pd_x^2 d_y)$ for each $\Frob_p(\omega_i)$. Every
single infinite reduction step takes time $\SoftOh(\log(p) d_x^2 nN)$, so all $\Frob_p(\omega_i)$ can be reduced in time 
\[
\SoftOh(g (pd_x^2 d_y) \log(p) d_x^2 nN) \subset \SoftOh(p d_x^5 d_y^2 nN).
\]
After this reduction procedure, we project from $E_0 \cap E_{\infty}$ onto the basis
$[\omega_1,\dotsc,\omega_{2g}]$ and read off the entries of $\Phi$. This involves computing $\BigOh(g)$ products of a 
vector by a matrix of size $\BigOh(d_x^2 d_y)$ (under Assumption~\ref{assump:Winfty}). Therefore, it can be done in time 
\[
\SoftOh(\log(p) g  (d_x^2 d_y)^2 nN) \subset \SoftOh(\log(p)  d_x^5 d_y^3 nN).
\] 
Combining all of this, the total time complexity of this step is
\[
\SoftOh(pd_x^4 d_y^2 n N^2+d_x^5 d_y^3 n N).
\]

\subsection{Step IV: Determine $Z(X,T)$} \mbox{ } \\

It follows from the Lefschetz formula for rigid cohomology that
\begin{align*}
Z(X,T)                   &= \frac{\chi(T)}{(1-T)(1-qT)}, 
\intertext{where} 
\chi(T) &= \det\bigl(1-\Frob_p^n T|\Hrig^1(X) \bigr).
\end{align*}
Since $\Frob_p$ is not linear but $\sigma$-semilinear, the matrix of $\Frob_p^n$ with respect to the basis $[\omega_1,\dotsc,\omega_{2g}]$
is given by
\[
\Phi^{(n)}=\Phi^{\sigma^{(n-1)}} \Phi^{\sigma^{(n-2)}} \dotsm \Phi. 
\]
Note that $\chi(T)$ is the reverse characteristic polynomial of $\Phi^{(n)}$. It is known (see for example \cite{pt}) that $\Phi^{(n)}$ 
can be computed from $\Phi$ in time $\SoftOh(\log^2(p) g^{\theta} n N )$ and that $\chi(T)$ 
can be computed from $\Phi^{(n)}$ in time $\SoftOh(\log(p)g^{\theta} nN )$.
Therefore, the total time complexity of this step is 
\[
\SoftOh(\log^2(p) g^{\theta} n N) \subset \SoftOh(\log^2(p)(d_x d_y)^{\theta} n N).
\]

\subsection{The $p$-adic precision} \mbox{ } \\

So far we have only obtained an approximation to $\chi(T)$, since 
we have computed to $p$-adic precision $N$. Moreover, because of loss of precision
in the computation, in general $\chi(T)$ will not even be correct to precision~$N$. 
So what precision~$N$ is sufficient to determine $\chi(T)$ exactly?

\begin{prop} \label{prop:prec}
The least $p$-adic precision $N$ that is sufficient to determine $\chi(T)$ satisfies $N \in \SoftOh(d_x d_y n)$.
\end{prop}

\begin{proof}
We assume for simplicity as in \cite{kedlaya} that $\ord_{p}(\Phi) \geq 0$. After the proof 
we will say something more about the general case.

It follows from the Weil conjectures that $\chi(T)$ is determined by the bottom half of its coefficients, all of which are
bounded in absolute value by $\binom{2g}{g} q^{\frac{g}{2}}$. Therefore, if $\chi(T)$ is known 
to $p$-adic precision at least $\lceil \log_p \bigl( 2 \binom{2g}{g} q^{\frac{g}{2}} \bigr) \rceil$, then it is 
determined exactly. Since $\ord_p(\Phi) \geq 0$, there will be no loss of precision in computing $\Phi^{(n)}$ and $\chi(T)$, so
that it is sufficient to compute $\Phi$ to $p$-adic precision $\lceil \log_p \bigl( 2 \binom{2g}{g} q^{\frac{g}{2}} \bigr) \rceil$.

From Proposition~\ref{prop:convbound} and formula~\eqref{eq:Fp1}, it follows that in equation~\eqref{eq:Fpomega} we have
$\max J \leq p(N-1)-1$.
Therefore, the loss of precision during the reductions at the points not lying over $x=\infty$ is at most
$\lfloor \log_p(p(N-1)e) \rfloor$
by Proposition~\ref{prop:finiteprecision}.

Similarly, the coefficients of $\Frob_p(y^i/r)$ with respect to the basis
$[b_0^{\infty},\dotsc, b_{d_x-1}^{\infty}]$ have order at $x=\infty$ at least 
$p(\ord_{\infty}((W^{\infty})^{-1})+\deg(r))$
by the proof of Proposition~\ref{prop:convbound}. 
It follows from formula~\eqref{eq:Fp1} and the definition of $E_{\infty}$ that the coefficients of 
$\Frob_p(\omega_i)$ with respect to the basis $[b_0^{\infty},\dotsc, b_{d_x-1}^{\infty}]$, which are elements of $\Omega^1(\mathbb{V})$ now, 
have order at $x=\infty$ at least
\begin{align*}
p \Bigl(\ord_{\infty}((W^{\infty})^{-1})+\deg(r) \Bigr) - (p-1) + p \Bigl(\ord_0(W^{\infty})-\deg(r)+2 \Bigr) - 2 \geq \\ 
p\Bigl( \ord_{\infty}((W^{\infty})^{-1})+ \ord_0(W^{\infty})\Bigr) - 1.
\end{align*}
Note that the reductions at the points not lying over $x=\infty$ can introduce poles at $x=\infty$, but these can be ignored
since they have order at $x=\infty$ at least 
\[
\ord_{\infty}((W^{\infty})^{-1}) \geq p\Bigl( \ord_{\infty}((W^{\infty})^{-1})+ \ord_0(W^{\infty})\Bigr) - 1,
\]
using that $\ord_{\infty}((W^{\infty})^{-1})$, $\ord_0(W^{\infty})$ are both negative.
Hence, when applying Proposition~\ref{prop:infiniteprecision} 
to the $1$-form that remains after the reductions at the points not lying over $x=\infty$, we have that
$m \leq -p\bigl( \ord_{\infty}((W^{\infty})^{-1})+\ord_0(W^{\infty}) \bigr)$.
Therefore, the loss of precision during the reductions at the points lying over $x=\infty$ is at most
\[\lfloor \log_p \Bigl(-p\bigl(\ord_{\infty}((W^{\infty})^{-1})+\ord_0(W^{\infty}) \bigr) e_{\infty} \Bigr) \rfloor.\]

By construction of our basis $[\omega_1,\dotsc,\omega_{2g}]$, there will be no further loss of precision computing the matrix $\Phi$. 
We conclude that it is sufficient for $N$ to satisfy 
\begin{align*} 
N - \lfloor \log_p(p(N-1)e) \rfloor - \lfloor \log_p \Bigl(-p\bigl(\ord_{\infty}((W^{\infty})^{-1})+\ord_0(W^{\infty}) \bigr) e_{\infty} \Bigr) \rfloor \geq \\
\lceil \log_p \bigl( 2 \binom{2g}{g} q^{\frac{g}{2}} \bigr) \rceil.
\end{align*}
From this it follows that $N \in \SoftOh(d_x d_y n)$ using Propositions~\ref{prop:degr} and~\ref{prop:ordbounds}.
\end{proof}

\begin{rem}
If we do not assume that $\ord_p(\Phi) \geq 0$, then we can use Propositions~\ref{prop:convbound},~\ref{prop:finiteprecision} and~\ref{prop:infiniteprecision} to obtain
a lower bound for $\ord_p(\Phi)$. Taking into account the extra loss of precision $(n-1)\ord_p(\Phi)$ for computing
$\Phi^{(n)}$ and $(2g-1)n\ord_p(\Phi)$ for computing $\chi(T)$, we still have that $N \in \SoftOh(d_x d_y n)$. 
However, a bound for $N$ obtained this way will not be very good in practice. One can obtain a much sharper bound for $\ord_p(\Phi)$ and the 
loss of precision in computing $\Phi^{(n)}$ and $\chi(T)$, using the existence of the $\Frob_p$-invariant $\ZZ_q$-lattice coming from the 
(log)-crystalline cohomology inside the rigid cohomology.
\end{rem}

\begin{thm} \label{thm:time}
The time complexity of the algorithm presented in this section is $\SoftOh(p d_x^6 d_y^4 n^3)$. 
\end{thm}

\begin{proof}
We take the sum of the complexities of the different steps using Proposition~\ref{prop:prec}, 
leaving out terms and factors that are absorbed by the $\SoftOh$. 
\end{proof}

For the analysis of the space complexity, we will not go into the same detail as for the time
complexity. However, using Assumption~\ref{assump:Winfty} at the same two points as in the 
analysis of the time complexity, one can prove the following theorem. 

\begin{thm} \label{thm:space}
The space complexity of the algorithm presented in this section is 
$\SoftOh(p d_x^4 d_y^3 n^3)$. 
\end{thm}

\begin{proof}
The space complexity of the algorithm turns out to be that of storing
a single $\Frob_p(\omega_i)$, or equivalently an element of $\mathcal{R}$, 
which is $\SoftOh(p d_x^2 d_y  \bigl(N+d_x \bigr) n N )$.
The result now follows using Proposition~\ref{prop:prec}.
\end{proof}

\begin{rem}
There are some standard ways to improve the algorithm from this section in practice:
\begin{enumerate}
\item We computed the Frobenius lift by working with $p$-adic precision $N_i=2^i$ in the $i$th step
of the Hensel lift. Setting $N_{\nu}=N$ and $N_{i-1}=\lceil N_i/2 \rceil$ for all $1 \leq i \leq \nu$, we 
still obtain the correct Frobenius lift to precision~$N$, while having to compute to lower precision in
every step. 
\item The bound $\log_p \bigl( 2 \binom{2g}{g} q^{\frac{g}{2}} \bigr)$ for the $p$-adic precision of $\chi(T)$ 
can be lowered using the Newton-Girard identities~\cite{newtongirard}.
\end{enumerate}
These improvements do not affect the complexity of the algorithm, but are important in practice.
\end{rem} 

\subsection{Our assumptions}

\subsubsection{Assumption 1}

Without this assumption, Theorem~\ref{thm:comparison} does not hold and we
cannot compute in $\Hrig^1(U)$ as in Section~\ref{sec:coho}. Therefore, Assumption~\ref{assump:goodlift}
is essential and cannot be lifted. It would be interesting to know under what conditions a lift satisfying
this assumption can be found. Note that for a smooth curve and a map $x$ to the projective line defined over a number field $K$,
Assumption~\ref{assump:goodlift} is satisfied at all but finitely many prime ideals of $\mathcal{O}_K$.

\subsubsection{Assumption 2}

This assumption serves to simplify the exposition and can be weakened as follows. Note that
Assumption~\ref{assump:Qsmooth} is equivalent to asking that $[y^0,\dotsc,y^{d_x-1}]$ is an integral basis for $\QQ_q(x,y)$ over 
$\QQ_q[x]$. Let us now assume instead that a matrix $W^0 \in Gl_{d_x}(\ZZ_q[x,1/r])$ is known such that if we denote
$b^0_j = \sum_{i=0}^{d_x-1} W^{0}_{i+1, j+1} y^i$ for all $0 \leq j \leq d-1$, then $[b_0^{0}, \dotsc, b_{d_x-1}^{0}]$
is an integral basis for $\QQ_q(x,y)$ over $\QQ_q[x]$. Then our algorithm should continue to work, extending it to
arbitrary curves for which we can find a lift that satisfies Assumption~\ref{assump:goodlift}. However, since quite a lot of small 
changes are needed in the different steps of the algorithm and, more importantly, we have not implemented this more general algorithm 
yet, for now we limit ourselves to the less general case. 

\subsubsection{Assumption 3}

Note that Assumptions~\ref{assump:Qsmooth} and~\ref{assump:infty} are in fact 
very similar: we need an integral basis for $\QQ_q(x,y)$ over both $\QQ_q[x]$ and $\QQ_q[x^{-1}]$. In both cases, algorithms like the 
one from \cite{vanhoeij} are available for computing the integral bases.

\subsubsection{Assumption 4}

This assumption is the least important of all the assumptions. We have used it a couple of times in the complexity analysis, to bound
the complexity of doing linear algebra in $E_0 \cap E_{\infty}$. Note that in Proposition~\ref{prop:triangle}, we have that $\ord_{\infty}(W^{\infty})=0$. 
For more general Newton polygons this also seems to be the case experimentally. With large random searches we have not been able to find a single example
satisfying $\ord_{\infty}(W^{\infty}) \leq -d_x d_y/2$. Therefore, we expect that Assumption~\ref{assump:Winfty} can be removed. 

\subsection{Implementation} \mbox{ } \\

We have implemented our algorithm in the computer algebra system Magma \cite{magma}. In examples where we can compare against 
either \cite{cdv} or \cite{walk}, our algorithm runs at least two orders of magnitude faster. The code can be found at \url{http://perswww.kuleuven.be/jan_tuitman}
and comes in two different packages: \verb{pcc_p{ for primefields and \verb{pcc_q{ for non-primefields. 
We give an example for each package below, mainly to demonstrate how to use the code. More examples and timings can be found in the example files that come with the packages. 
The computations were carried out with Magma v2.20-3 on a 3.0GHz Intel Core i7-3540M processor. 

\subsubsection*{Example 1} A random curve over $\FF_{11}$ with $d_x=4$ and $d_y=5$ (genus $12$). \medskip
\begin{Verbatim}[fontsize=\tiny]
load "pcc_p.m";
Q:=y^4+(6*x^5+10*x^4+8*x^3+5*x^2+7*x+5)*y^3+(4*x^5+x^4+8*x^3+6*x^2+6*x)*y^2+(3*x^5+5*x^4+9*x^3+2*x^2+10*x+4)*y
+6*x^5+3*x^4+7*x^3+10*x^2+4*x+3;
p:=11;
N:=9;
chi:=num_zeta(Q,p,N:verbose:=true);
\end{Verbatim}

\noindent The input consists of the polynomial $\mathcal{Q} \in \mathbb{Z}[x,y]$, the prime $p$ and the $p$-adic working 
precision $N$. The output is the numerator $\chi(T)$ of the zeta function $Z(X,T)$. In this case it is\medskip
\begin{Verbatim}[fontsize=\tiny]
3138428376721*T^24-285311670611*T^23-233436821409*T^22+80170221494*T^21-20364093695*T^20+3799998345*T^19
+2657341500*T^18-754684986*T^17+182500065*T^16-37234725*T^15-9607037*T^14+6197609*T^13-939504*T^12+563419*T^11
-79397*T^10-27975*T^9+12465*T^8-4686*T^7+1500*T^6+195*T^5-95*T^4+34*T^3-9*T^2-T+1. 
\end{Verbatim}
The computation took 27.9s and less than 32MB of memory.

\subsubsection*{Example 2} A random curve over $\FF_{7^{10}}$ with $d_x=3$ and $d_y=5$ (genus $8$). \medskip
\begin{Verbatim}[fontsize=\tiny]
load "pcc_q.m";
Q:=y^3+((a^9+5*a^7+3*a^5+6*a^4+4*a^3+2*a^2+5*a+1)*x^5+(5*a^9+5*a^8+2*a^7+2*a^6+3*a^5+a^4+6*a^3+4*a+4)*x^4
+(2*a^9+6*a^7+6*a^6+2*a^5+6*a^4+5*a^3+6)*x^3+(3*a^9+2*a^8+3*a^7+3*a^6+a^5+4*a^4+5*a^3+4*a^2+3*a+3)*x^2 
+(5*a^9+3*a^8+a^7+2*a^6+4*a^5+a^4+3*a^3+5*a^2+2)*x+(4*a^8+2*a^7+4*a^6+a^4+4*a^3+a^2+2*a+4))*y^2
+((2*a^9+3*a^8+3*a^7+6*a^6+6*a^5+6*a^4+4*a^3+5*a^2+6*a)*x^5+(5*a^9+3*a^8+2*a^6+2*a^5+4*a^4+2*a^3+4*a^2+3*a+6)*x^4 
+(3*a^9+3*a^8+6*a^7+5*a^6+3*a^5+3*a^4+5*a^3+4*a^2+4*a+1)*x^3+(2*a^9+2*a^8+5*a^7+5*a^6+5*a^5+6*a^4+a^3+a^2+2*a+2)*x^2 
+(3*a^8+3*a^6+3*a^5+5*a^3+4*a^2+4*a+2)*x+(4*a^9+2*a^8+5*a^7+5*a^6+2*a^5+5*a^4+6*a^3+2*a+4))*y 
+(a^9+a^8+2*a^7+4*a^6+2*a^5+a^4+2*a^3+4*a^2+6*a+2)*x^5+(4*a^9+5*a^7+a^6+a^5+3*a^4+2*a^3+6*a+6)*x^4
+(4*a^9+4*a^8+4*a^7+a^6+a^5+5*a^4+2*a^3+a^2+2*a)*x^3+(5*a^9+5*a^7+6*a^6+3*a^5+6*a^4+4*a^3+3*a^2+6*a)*x^2
+(5*a^8+2*a^7+2*a^6+3*a^2+a)*x+a^9+6*a^8+6*a^7+2*a^6+6*a^5+4*a^4+3*a^3+5*a+2;
p:=7;
n:=10;
N:=45;
chi:=num_zeta(Q,p,n,N:verbose:=true);
\end{Verbatim}

\noindent The input consists of the polynomial $\mathcal{Q} \in \ZZ[a][x,y]$, the prime $p$, the extension degree $n$ and the $p$-adic working precision $N$.
Here $a$ represents a standard generator for $\ZZ_q/\ZZ_p$, i.e. it is a root of a Conway polynomial. The output is the numerator $\chi(T)$ of the zeta function $Z(X,T)$. 
In this case it is\medskip
\begin{Verbatim}[fontsize=\tiny]
40536215597144386832065866109016673800875222251012083746192454448001*T^16
+734594936640916515108002147869799216237456127361200615126315631*T^15
+37833822114992619972303659616442535094177702647200606500823*T^14
+2969545553762454604862263614126054405430871338256835484*T^13+323896800674094517822826810513267326953587001034849*T^12
+22636175881373275379227578482427791310493422448*T^11+146359712260050195498039226426210033108323*T^10
+66506665686156219471818560867075857462*T^9+3128031304748736252054098124793644*T^8+235442453530348846499533702038*T^7
+1834259371881387520432323*T^6+1004296292146625341552*T^5+50872731607858849*T^4+1651155559516*T^3+74472823*T^2+5119*T+1.
\end{Verbatim}
The computation took 2458s and about 350MB of memory.

\bibliographystyle{alpha}
\bibliography{curves}

\begin{thebibliography}{CDV06}

\bibitem[BC94]{baldachiar}
Francesco Baldassarri and Bruno Chiarellotto.
\newblock Algebraic versus rigid cohomology with logarithmic coefficients.
\newblock In {\em Barsotti {S}ymposium in {A}lgebraic {G}eometry ({A}bano
  {T}erme, 1991)}, volume~15 of {\em Perspect. Math.}, pages 11--50. Academic
  Press, San Diego, CA, 1994.

\bibitem[BCP97]{magma}
Wieb Bosma, John Cannon, and Catherine Playoust.
\newblock The {M}agma algebra system. {I}. {T}he user language.
\newblock {\em J. Symbolic Comput.}, 24(3-4):235--265, 1997.
\newblock Computational algebra and number theory (London, 1993).

\bibitem[BP00]{beelenpellikaan}
Peter Beelen and Ruud Pellikaan.
\newblock The {N}ewton polygon of plane curves with many rational points.
\newblock {\em Des. Codes Cryptogr.}, 21(1-3):41--67, 2000.
\newblock Special issue dedicated to Dr. Jaap Seidel on the occasion of his
  80th birthday (Oisterwijk, 1999).

\bibitem[CDV06]{cdv}
W.~Castryck, J.~Denef, and F.~Vercauteren.
\newblock Computing zeta functions of nondegenerate curves.
\newblock {\em IMRP Int. Math. Res. Pap.}, pages Art. ID 72017, 57, 2006.

\bibitem[DV06a]{dv}
Jan Denef and Frederik Vercauteren.
\newblock Counting points on {$C_{ab}$} curves using {M}onsky-{W}ashnitzer
  cohomology.
\newblock {\em Finite Fields Appl.}, 12(1):78--102, 2006.

\bibitem[DV06b]{dvhyp}
Jan Denef and Frederik Vercauteren.
\newblock An extension of {K}edlaya's algorithm to hyperelliptic curves in
  characteristic 2.
\newblock {\em J. Cryptology}, 19(1):1--25, 2006.

\bibitem[GG01]{gaugu}
Pierrick Gaudry and Nicolas G{\"u}rel.
\newblock An extension of {K}edlaya's point-counting algorithm to superelliptic
  curves.
\newblock In {\em Advances in cryptology---{ASIACRYPT} 2001 ({G}old {C}oast)},
  volume 2248 of {\em Lecture Notes in Comput. Sci.}, pages 480--494. Springer,
  Berlin, 2001.

\bibitem[Har07]{harvey1}
David Harvey.
\newblock Kedlaya's algorithm in larger characteristic.
\newblock {\em Int. Math. Res. Not. IMRN}, (22):Art. ID rnm095, 29, 2007.

\bibitem[Har14]{harvey2}
David Harvey.
\newblock Counting points on hyperelliptic curves in average polynomial time.
\newblock {\em Ann. of Math. (2)}, 179(2):783--803, 2014.

\bibitem[Hub10]{hubrechts}
Hendrik Hubrechts.
\newblock Fast arithmetic in unramified {$p$}-adic fields.
\newblock {\em Finite Fields Appl.}, 16(3):155--162, 2010.

\bibitem[Ked01]{kedlaya}
Kiran~S. Kedlaya.
\newblock Counting points on hyperelliptic curves using {M}onsky-{W}ashnitzer
  cohomology.
\newblock {\em J. Ramanujan Math. Soc.}, 16(4):323--338, 2001.

\bibitem[Ked08]{newtongirard}
Kiran~S. Kedlaya.
\newblock Search techniques for root-unitary polynomials.
\newblock In {\em Computational arithmetic geometry}, volume 463 of {\em
  Contemp. Math.}, pages 71--81. Amer. Math. Soc., Providence, RI, 2008.

\bibitem[KT12]{kedlayatuitman}
Kiran~S. Kedlaya and Jan. Tuitman.
\newblock Effective convergence bounds for {F}robenius structures on
  connections.
\newblock {\em Rend. Semin. Mat. Univ. Padova.}, pages 7--16, 2012.

\bibitem[Lau06]{lauder}
Alan G.~B. Lauder.
\newblock A recursive method for computing zeta functions of varieties.
\newblock {\em LMS J. Comput. Math.}, 9:222--269, 2006.

\bibitem[PT13]{pt}
Sebastian Pancratz and Jan Tuitman.
\newblock Improvements to the deformation method for counting points on smooth
  projective hypersurfaces.
\newblock {\em preprint}, 2013.
\newblock http://arxiv.org/abs/1307.1250.

\bibitem[Sto00]{storjohann}
Arne Storjohann.
\newblock {\em Algorithms for Matrix Canonical Forms}.
\newblock PhD thesis, Swiss Federal Institute of Technology -- ETH, 2000.

\bibitem[vH94]{vanhoeij}
Mark van Hoeij.
\newblock An algorithm for computing an integral basis in an algebraic function
  field.
\newblock {\em J. Symbolic Comput.}, 18(4):353--363, 1994.

\bibitem[Wal10]{walk}
George Walker.
\newblock {\em Computing zeta functions of varieties via fibration}.
\newblock PhD thesis, Oxford, 2010.

\bibitem[Wil12]{williams2012}
Virginia~Vassilevska Williams.
\newblock Multiplying matrices faster than {C}oppersmith-{W}inograd [extended
  abstract].
\newblock In {\em S{TOC}'12---{P}roceedings of the 2012 {ACM} {S}ymposium on
  {T}heory of {C}omputing}, pages 887--898. ACM, New York, 2012.

\end{thebibliography}

\end{document}